\documentclass [14pt]{amsart}
\usepackage{amsmath}
\usepackage{amssymb}
\usepackage{amscd}
\usepackage{epsfig}
\usepackage{amsxtra}
\usepackage{pifont}
\usepackage{mathabx}
\usepackage{MnSymbol}
\usepackage{accents}
\usepackage{scalerel,stackengine}
\usepackage{xcolor}
\stackMath
\newcommand\reallywidehat[1]{%
\savestack{\tmpbox}{\stretchto{%
  \scaleto{%
    \scalerel*[\widthof{\ensuremath{#1}}]{\kern-.6pt\bigwedge\kern-.6pt}%
    {\rule[-\textheight/2]{1ex}{\textheight}}
  }{\textheight}%
}{0.5ex}}%
\stackon[1pt]{#1}{\tmpbox}%
}

\newcommand\reallywidecheck[1]{%
\savestack{\tmpbox}{\stretchto{%
  \scaleto{%
    \scalerel*[\widthof{\ensuremath{#1}}]{\kern-.6pt\bigwedge\kern-.6pt}%
    {\rule[-\textheight/2]{1ex}{\textheight}}
  }{\textheight}%
}{0.5ex}}%
\stackon[1pt]{#1}{\scalebox{-1}{\tmpbox}}%
}

\numberwithin{equation}{section}

\newcommand{\ZZ}{{\mathbb Z}}
\newcommand{\CC}{{\mathbb C}}

\newcommand{\TT}{\mathbb T}
\newcommand{\NN}{\mathbb N}

\newcommand{\XX}{\mathbb X}

\newcommand{\cL}{{\mathcal L}}

\newcommand{\Per}{\mbox{Per} }

\newcommand{\cM}{{\mathcal M}}

\newcommand{\cA}{{\mathcal A}}
\newcommand{\cB}{{\mathcal B}}

\newcommand{\mm}{\mathsf{m}}

\newcommand{\dd}{\mbox{\rm d}}

\newcommand{\Cc}{C_{\mathsf{c}}}

\newcommand{\bap}{Bap}

\newcommand{\SSS}{\mathbb S}

 \newtheorem{theorem}{Theorem}[section]
 \newtheorem{lemma}[theorem]{Lemma}
 \newtheorem{prop}[theorem]{Proposition}
 \newtheorem{corollary}[theorem]{Corollary}

 \newtheorem{definition}[theorem]{Definition}
 \newtheorem{example}[theorem]{Example}
  \newtheorem{remark}[theorem]{Remark}

\allowdisplaybreaks

\begin{document}
\title{Wiener--Wintner points for topological dynamical systems }
\author{Daniel Lenz}
\address{Mathematisches Institut, Friedrich Schiller Universit\"at Jena, 07743 Jena, Germany}
\email{daniel.lenz@uni-jena.de}
\urladdr{http://www.analysis-lenz.uni-jena.de}

\author{Nicolae Strungaru}
\address{Department of Mathematics and Statistics, MacEwan University \\
10700 -- 104 Avenue, Edmonton, AB, T5J 4S2, Canada\\
and \\
Institute of Mathematics ``Simon Stoilow''\\
Bucharest, Romania}
\email{strungarun@macewan.ca}
\urladdr{https://sites.google.com/macewan.ca/nicolae-strungaru/home}

\begin{abstract}
We consider measurable and topological dynamical systems over locally compact abelian groups. Our main observation  relates convergence  of Wiener-Wintner type averages to
eigenvalues of the dynamical system in question. As a consequence we infer
 existence of Fourier--Bohr coefficients for all characters for a set of points satisfying a specific genericity condition. In the topological case this leads naturally to the concept of what we call Wiener--Wintner point and we present a thorough study of such points. In particular we show that they have full measure in the ergodic case, and we relate them to Besicovitch almost periodicity. For dynamical systems of translation bounded measures, which are the  crucial models in aperiodic order, our results give that the Wiener--Wintner points are exactly the points allowing for a diffraction theory with the consistent phase property.
\end{abstract}

\maketitle

\section{Introduction}
We consider Wiener--Wintner type  averages for (continuous) functions on
dynamical systems over locally compact Abelian groups.  In the
simplest situation we deal with a compact metric space $X$, together with a
homeomorphism  $T  : X\longrightarrow X$ and a $T$-invariant
probability measure $\mm$, and we study points $x\in X$ for which the
limits
\begin{equation}\label{eq2}
\lim_N \frac{1}{2N+1} \sum_{n=-N}^N f(T^n x) z^n \,,
\end{equation}
exist for all continuous $f : X\longrightarrow \CC$ and a given  fixed
$z\in \TT$. Here, $\TT$ denotes the multiplicative group
$$\TT:=\{w\in\CC: \mbox{ with }  |w| =1\} \,.$$

Our basic observation relates existence of these limits (for all $f\in
C(X)$)  to existence of an eigenfunction to $z$ (see Lemma
\ref{lem-eigen-exi} for a precise statement in a substantially more
general situation). This allows us to establish vanishing (and, in
particular,  existence) of the limit for compact subsets of $\TT$ not
meeting any eigenvalue for all generic points $x\in X$ (Proposition~
\ref{uniform-vanishing}). In the uniquely ergodic case this  also
allows us to infer uniform (in $x\in X$) vanishing of the limit for
compact sets in $\TT$ outside of eigenvalues, thereby generalizing
results of Lenz \cite{Len}, which in turn generalize results of Assani
\cite{Ass} (see  \cite{Rob,Sch} as well). All of this is
discussed in Section~\ref{sect:top}.

Then, in Section \ref{sec-Strongly}, we study those points
where the above limits exist for all continuous $f \in C(X)$ and all $z\in
\TT$. We call these Fourier--Bohr points. By our basic observation any
Fourier--Bohr point comes with a family of eigenfunctions. It is not
clear that this family is an orthonormal basis of the space of
eigenfunctions. If this is the case, we call the corresponding point a
Wiener--Wintner point. The study of Wiener--Wintner points is the main
thrust of the second half of the article. We show that existence of
such points implies ergodicity and that ergodicity implies that almost
all points are Wiener--Wintner points. Moreover, we show that for dynamical systems with pure
point spectrum, the
Wiener--Wintner points are exactly the generic points that are Besicovitch almost
periodic. This ties in with recent work of Kasjan and Keller
\cite{KK} focusing on the generic Besicovitch almost periodic points.

Our interest in generic Besicovitch  almost periodic points comes from
aperiodic order also known as mathematical theory of quasicrystals,
see \cite{TAO}. A key feature of aperiodic order  is pure point
diffraction. In the mathematical modelling this leads to dynamical
systems with discrete spectrum. The Besicovitch almost periodic points
are then exactly the points allowing for a convincing diffraction
theory \cite{LSS,LSS2}. This is discussed in the last section of the article.

Our results in the topological case are an  applications of a more
general approach working for measurable dynamical systems as
developed in  Section \ref{sect:main} after the general introduction to
dynamical systems in Section \ref{sec-Basics}.  This approach is
centered around two concepts we introduce. These are the  concept of
genericity of a point for a subspace and the concept of core for a
subspace. Our results in the measurable case give the classical  Wiener--Wintner
theorem as a  by product. The application to the topological case then
comes about by choosing $C(X)$ as the subspace in question.

The results of this article are used in Chapter 8 of the treatise \cite{LSS}.

\section{Basics on  measurable dynamical systems}\label{sec-Basics}
Let $G$ be a $\sigma$-compact, locally compact Abelian
group $G$ (LCAG). The composition
on $G$ itself is written additively. We fix a Haar measure  on $G$. The Haar measure
of a measurable subset $A \subseteq  G$ is denoted by $|A|$ and the
integral of an integrable function $f$ on $G$ by $\int_{G} f(t)\, \dd t$.

Whenever the measurable  space $(X,\Sigma)$ is equipped with a measurable action
$$
\alpha: G\times X\longrightarrow X , \qquad (t,x)\mapsto
\alpha_t (x) \,,
$$
we call $(X,\Sigma, G)$ a \textit{measurable dynamical system (over the
space $X$)} and  write $t.x$
instead of $\alpha_t (x)$ for $t\in G$ and $x\in X$.
The \textit{orbit} of $x$ under the action is given by
$$\mbox{O}(x) :=\{ t.x : t\in G\}.$$
For a function $f : X\longrightarrow \CC$  and $x\in X$ we then define the \textit{coding} of $f$ on the orbit of $x$ to be the function
$f_x : G\longrightarrow \CC$ defined by
\[
f_x (t):= f(t.x) \,.
\]
For a function $f : G \to \CC$ and some $t \in G$, we denote by $\tau_t$ the translation
\[
(\tau_tf)(x):= f(x-t) \,.
\]
Similarly, for $g: X \to \CC$ and $t \in G$ we denote by
$\tau_t g$ the translation
\[
(\tau_tf)(x):= g(t.x) \,.
\]
It is easy to see that for all $g: X \to \CC$, all $x \in X$ and all $t \in G$ we have
\[
\tau_t(g_x)= (\tau_{-t}g)_x \,.
\]

A measure  $\mm$ on $(X,\Sigma)$ is called \textit{invariant} if $\mm(\alpha_{t} A) = \mm(A)$ holds for all $A\in \sigma$ and all $t\in G$.

Given a $G$-invariant measure $\mm$ on $(X, \Sigma)$, we refer to the triple  $(X,G,\mm)$ as measurable dynamical system.
Invariance of $\mm$  implies
$$\int_X f(t.x)\dd \mm(x) = \int_X f(y) \dd \mm(y)$$
for all $t\in G$ and all measurable positive $f$ on $X$.

For $1\leq p <\infty$ we denote by $\mathcal{L}^p (X,m)$ the set of measurable $f : X\longrightarrow \CC$ with
$$\int_X |f(x)|^p \dd \mm(x) <\infty \,.
$$
Identifying functions which agree almost everywhere we then arrive at the Banach spaces $L^p (X,\mm)$. We will mostly deal with $L^2 (X,\mm)$, which is a Hilbert space with inner product
$$\langle f,g\rangle =\int_X f(x) \cdot \overline{g(x)} \dd \mm(x) \,,$$
and associated norm
$$\|f\|_2:= \langle f, f\rangle^{1/2} \,,$$
where we followed the custom and wrote $f,g$ for the classes represented by the functions $f$ and $g$.

In our considerations below we will have to be more careful at certain places and restrict attention to elements of $\mathcal{L}^2 (X,\mm)$ rather than of $L^2 (X,\mm)$  as we are interested in specific sets of points, where suitable statements hold (or do not hold).

Any invariant  measure $\mm$  on $X$ gives rise to a family of unitary operators $U_t, t\in G,$ on the space $L^2 (X,\mm)$ via
\begin{align*}
&U_t : L^2 (X,\mm)\longrightarrow L^2 (X,\mm) \\
&(U_t f) (x) = f(t.x) \,,
\end{align*}
for $t\in G$. As $\alpha$ is an action, $(U_t)_{t \in G}$ form a unitary representation i.e.
\begin{align*}
U_0 &= \mbox{Id} \qquad  \mbox{ and } \\
U_{t+s} &= U_{t} U_s \qquad \forall t,s\in G \,.
\end{align*}

\smallskip

A  function $f$ on $X$  is \textit{$G$-invariant}(or simply \textit{invariant}) if for any $t\in G$ the functions $\tau_tf$ and $f$ agree almost everywhere. The dynamical system $(X,G,\mm)$ is \textit{ergodic} if any invariant measurable complex valued function is constant almost everywhere.

\medskip

Recall next that a \textit{ F\o lner sequence} $(A_n)$ in $G$, is a sequence of open relatively compact subsets $A_n$ of $G$  with the property that for all $t \in G$ we have
\[
\lim_n \frac{|(A_n \triangle (t+A_n))|}{|A_n|} =0 \,.
\]
A F\o lner sequence exists in a LCAG $G$ if and only if $G$ is $\sigma$-compact \cite[Prop.~B6]{SS}. For this reason, we will always assume that $G$ is $\sigma$-compact.
A F\o lner sequence $\cA$  is \textit{tempered} if there exists a $C>0$ with
  \[
  \left|\bigcup_{k=1}^{n-1} (B_n-B_k) \right| \leq C |B_n| \,,
  \]
for all $n\in \NN$.   Some basic properties of tempered F\o lner sequences to be used subsequently are discussed in the  Appendix \ref{tempered}.
We have the following pointwise ergodic theorem.

\begin{theorem}[Pointwise ergodic theorem]\label{thm-point-erg}\cite[Theorem 1.2 and Proposition 1.4]{Lin} Let $G$ be any $\sigma$-compact LCAG.  If $(X, G, \mm)$ is any ergodic dynamical system and $\cA=\{ A_n \}$ is any tempered F\o lner sequence then the \textit{pointwise ergodic theorem holds}, i.e. for each $f \in \cL^1(X,\mm)$
  there exists a set $X_0 \subseteq X$ with $\mm(X_0)=1$ such that for all $x \in X_0$ we have
\[
\lim_{n\to \infty} \frac{1}{|A_n|} \int_{A_n} f(t.x)\, \dd t = \int_X f(y) \,
\dd \mm(y) \,.
\]

Moreover, any F\o lner sequence contains a subsequence, which is tempered, and  along which accordingly  the pointwise ergodic theorem holds.
\end{theorem}

The dual group of $G$, i.e. the set of all continuous group homomorphisms from $G$ to the unit circle $\TT$ is denoted by $\widehat{G}$.
It is a group with addition defined by pointwise multiplication of characters, and it becomes a LCAG under a certain topology (see \cite[Sect~1.2.6]{Rud} for details).

Whenever $(X,G,\mm)$ is a dynamical system we call an $f\in L^2 (X,\mm)$ with $f\neq 0$ an \textit{eigenfunction} to the \textit{eigenvalue}   $\chi : G\longrightarrow  \CC$ if
$$f(t\cdot) = \chi (t) f$$
holds for all $t\in G$. It is not hard to see that $\chi$ must then belong to $\widehat{G}$.   The set of all eigenfunctions to a given eigenvalue $\chi$ together with the constant function $0$ is a closed subspace of $L^2 (X,\mm)$ called  the \textit{eigenspace} to $\chi$.
We call
\[
\SSS:= \{ \chi \in \widehat{G} : \chi \mbox{ is an eigenvalue for } (X,G,\mm) \} \ ,
\]
the \textit{point spectrum} of $(X, G, \mm)$. Moreover, we write $L(X, \mm)_{\mathsf{pp}}$ for the closure in $L^2(X,\mm)$ of the linear span of the eigenfunctions.

\medskip

For a given F\o lner sequence $(A_n)$, $\chi \in\widehat{G}$ and a measurable function $h : G\longrightarrow \CC$ we define
$$a_\chi^{A_n} (h):=\frac{1}{|A_n|} \int_{A_n} h(t) \cdot \overline{\chi(t)} \dd t \mbox{  (provided  the integral exists)}$$
and we define
$$a_\chi^{\cA} (h):=\lim_n a_\chi^{A_n} (h) \mbox{  (provided  the limit exists)} \,.$$
Whenever  the limit $a_{\chi}^\cA(h)$ exists, it is called the \textit{Fourier--Bohr coefficient} of $h$ at $\chi$.

\smallskip
For  a given $f\in \mathcal{L}^1(X,\mm)$, a F\o lner sequence $(A_n)$, some  $\chi \in \widehat{G}$  and $x\in X$  we can then consider
$a_\chi^{A_n} (f_x)$ and $a_\chi^{\mathcal{A}} (f_x)$ (provided these expressions are defined) and  we  call
$$a_\chi^{\cA} (f_x) = \lim_n \frac{1}{|A_n|}\int_{A_n} f(t.x) \cdot  \overline{\chi(t)} \dd t$$
the \textit{Fourier--Bohr coefficient} of $f$ at $\chi$  in $x\in X$.

The pointwise ergodic theorem can then be understood to deal with existence of Fourier--Bohr coefficients for $\chi =1$. The case  of general $\chi$  is the topic of this article. Here, we  already note that, by von Neumann ergodic theorem,
the averages
$$x\mapsto \frac{1}{|A_n|} \int_{A_n} f(t.x) \cdot \overline{ \chi (t)} \dd t$$
converge to $P_\chi f$ in  $L^2 (X,\mm)$, where $P_\chi$ is the projection onto the eigenspace to $\chi$.

\section{The main observation}\label{sect:main}
In this section we present our main observation. We start by introducing the following definitions.

\bigskip

\begin{definition}[$(E,G')$-generic points] Let $(X,G, \mm)$ be a measurable dynamical system and $\mathcal{A}$ a F\o lner sequence on $G$.  Let $E$ be a subset of $\mathcal{L}^2 (X,\mm)$ and $G'$ a subgroup of $G$. An element $x\in X$ is called {\it $(E,G')$-generic (with respect to $\mathcal{A}$)} if the following holds:
\begin{itemize}
\item  $\int_X f(y) \overline{g(y)} \dd \mm(y) = \lim_n \frac{1}{|A_n|} \int_{A_n}  f(t.x) \overline{g(t.x)} \dd t$
holds for all $f,g\in E$; and
\item $\lim_n \frac{1}{|A_n|} \int_{A_n \triangle (A_n+s)} |f(t.x)|^2 \dd t = 0$ for each  $f\in E$ and all $s\in G'$.
\end{itemize}
\end{definition}

\begin{remark} If all elements in $E$ are bounded functions, then for any subgroup $G' \leq G$ an element $x \in X$ is $(E,G')$-generic if and only if
\[
\lim_n \frac{1}{|A_n|} \int_{A_n}  f(t.x)\cdot \overline{g(t.x)} \dd t \, = \, \int_X f(y)\cdot \overline{g(y)} \dd \mm(y) \,,
\]
holds for all $f,g \in E$.
\end{remark}

In  the situation of the definition an $(E,G')$-generic point is also generic for the linear span of $E$. Thus, we can always assume without loss of generality that $E$ is a subspace when dealing with generic points. We also note that genericity is preserved under uniform limits.  Here,  the function $f$ on $X$ is the uniform limit of the sequence of functions $(f_n)$ if
$$\|f - f_n\|_\infty\to 0,n \to \infty,$$
with $\|g\|_\infty :=\sup_{x\in X} |g(x)|$.
Thus, we can also always assume that $E$ is closed under uniform limits.

\begin{definition}[Core]
Let $E$ be a subspace of $\mathcal{L}^2 (X, \mm)$. A subset $F \subseteq E$  is called a {\rm core for $E$} if any element of $E$ is a uniform limit of a sequence in the linear span of $F$.
\end{definition}

We will be particularly interested in subspaces with  a countable core.  The reason is that for such subspaces we can show under an ergodicity assumption  that generic points have full measure.
Specifically, the following is true.

\begin{prop}[Existence of generic elements]\label{existence-generic} Let $(X,G, \mm)$ be  a measurable  dynamical system and let $\cA=\{A_n\}$ be a F\o lner sequence.

Let $E$ be a subspace of $\mathcal{L}^2 (X, \mm)$ with a countable core and let $G'$ be a countable subgroup of $G$.  Then, the following holds:

\begin{itemize}
\item[(a)]  If $\mm$ is ergodic and $\cA$ is tempered then the set of generic elements for $(E,G')$ has full measure.

\item[(b)] If $\mm$ is ergodic, then there exists a subsequence $\cB$ of $\cA$ such that the set of generic elements for $(\mm,\cB)$ has full measure.
\end{itemize}
\end{prop}
\begin{proof}
\textbf{(a)}
 Let $D$ be a countable core of $E$. By the pointwise ergodic theorem, for each $f,g \in D$ there exists some set
$X_{f,g} \subseteq X$ of full measure such that
$$\lim_n \frac{1}{|A_n|} \int_{A_n} f(t.x) \overline{g(t.x)} \dd t = \langle f, g\rangle$$
holds for $f$ and $g$  and all $x \in X_{f,g}$.

By Lemma~\ref{Folner2}, for each $s \in G$, the sequences $(B_n)$ and $(C_n)$ with
\begin{align*}
  B_n & =A_n \cap (s+A_n)\\
  C_n & =A_n \cup (s+A_n)
 \end{align*}
 are tempered F\o lner sequences. Therefore, for each $f \in D$ there exists some $X_{f,s} \subseteq X$ of full measure such that
\begin{align*}
\lim_n \frac{1}{|B_n|} \int_{B_n} |f(t.x)|^2  \dd t = \| f \|_2^2  \\
\lim_n \frac{1}{|C_n|} \int_{C_n} |f(t.x)|^2  \dd t = \| f \|_2^2 \,,
\end{align*}
holds for $f$  and all $x \in X_{f,s}$.
It follows that the set
\[
Y:= \left( \bigcap_{f,g \in D} X_{f,g} \right) \cap \left( \bigcap_{f \in D, s \in G' } X_{f,s} \right)
\]
has full measure. For  $f,g \in D, s \in G'$ and $x \in Y$ we then infer
 $$
\lim_n \frac{1}{|A_n|} \int_{A_n} f(t.x) \overline{g(t.x)} \dd t = \langle f, g\rangle$$
and   with  to $A_n\triangle (A_n + s) = C_n \setminus B_n$  we also conclude
\begin{align*}
\lim_n \frac{1}{|A_n|} \int_{A_n \triangle (A_n+s)} |f(t.x)|^2 \dd t &= \lim_n \frac{1}{|C_n|} \int_{C_n} |f(t.x)|^2 \dd t -\lim_n \frac{1}{|B_n|} \int_{B_n} |f(t.x)|^2 \dd t \\
&=  \| f \|_2^2 - \| f \|_2^2\\
&=0 \,.
\end{align*}
This proves (a).

\smallskip

\textbf{(b)}  This follows from (a) and Theorem~\ref{thm-point-erg}.
\end{proof}

The following lemma  is the main ingredient for our subsequent considerations.

\begin{lemma}[Main observation: forcing  eigenvalues] \label{lem-eigen-exi} Let $(X,G, \mm)$ be a measurable dynamical system, $\cA$ a F\o lner sequence and $G'$ a dense subgroup of $G$.

Let $E\subset \mathcal{L}^2 (X,\mm)$ be a $G'$-invariant subspace and let $x\in X$ be $(E,G')$-generic. Let $\bar{E}$ be the closure of $E$ in $L^2(X, \mm)$.

Let $(\chi_n)$ be a sequence in $\widehat{G}$  converging pointwise to $\chi \in \widehat{G}$ such that
$$ F(g):= \lim_n\frac{1}{|A_n|} \int_{A_n} g(t.x) \cdot \overline{\chi_n(t)} \dd t$$
exists for all $g$ in $E$. Then, there exists a unique  $f_\chi \in \bar{E}$  satisfying
\begin{itemize}
  \item{} $\| f_\chi \|_2 \leq 1$;
  \item{} $U_t f_\chi = \chi(t) f_\chi \,$ for all $t\in G$.
\end{itemize}
Specifically, $f_\chi =0$ if $F(g) = 0$ for all $g\in E$ and  $f_\chi$ is an eigenfunction to $\chi$ otherwise.
\end{lemma}
\begin{proof} Clearly, the map
 $F$ is  linear.  Moreover, for  $g \in E$ and all $n$, by the Cauchy--Schwarz inequality we have
\begin{eqnarray*}
 \left| \frac{1}{|A_n|} \int_{A_n} g(t.x) \overline{\chi_n(t)} \dd t \right| &\leq& \sqrt{ \frac{1}{|A_n|} \int_{A_n} \left|  g(t.x) \right|^2 \dd t} \sqrt{\frac{1}{|A_n|} \int_{A_n} \left|  \chi_n(t) \right|^2 \dd t} \\
  &= & \sqrt{ \frac{1}{|A_n|} \int_{A_n} \left|  g(t.x) \right|^2 \dd t}\\
(x  \mbox{ generic }) 	&\to & \|g\|_2.\\
\end{eqnarray*}
Therefore, we find
\[
|F(g)| =\lim_n \left| \frac{1}{|A_n|} \int_{A_n} g(t.x) \overline{\chi_n(t)} \dd t \right| \leq \|g\|_2 \ .
\]
This means that $F$ is a linear bounded operator on   $E$. Therefore, $F$ can be extended to a bounded operator $F : \overline{E} \to \CC$. By the Riesz- representation theorem, there exists some $f_\chi \in \overline{E}$ such that, for all $g \in E$ we have
\[
F(g)= \langle g, f_\chi \rangle \ .
\]
Since $\|F \|\leq 1$ we get  $\|f_\chi \|_2 \leq 1$.
Moreover,  for all $g \in E$ and $s \in G'$ we have
\begin{eqnarray*}
\langle g, U_{s} f_\chi\rangle &=& \langle U_{-s} g, f_\chi\rangle\\
&=&  F(U_{-s} g)\\
 &=&  \lim_n \frac{1}{|A_n|} \int_{A_n} \tau_{-s} g(t.x) \overline{\chi_n(t)} \dd t\\
&=&  \lim_n \overline{\chi (s)} \frac{1}{|A_n|} \int_{A_n} g((t-s).x) \overline{\chi_n (t-s)} \dd t \\
&=& \lim_n \overline{\chi (s)} \frac{1}{|A_n  |} \int_{A_n-s} g(t.x) \overline{\chi_n (t)} \dd t\\
(\mbox{$x$ generic})\;\: &=& \lim_n \overline{\chi (s)} \frac{1}{|A_n |} \int_{A_n} g(t.x) \overline{\chi_n (t)} \dd t\\
 &=& \overline{\chi(s)}  F(g)\\
&=& \langle g, \chi(s) f_\chi\rangle.
\end{eqnarray*}
This easily gives
$$U_s f_\chi = \chi(s) f_\chi$$
for all $s\in G'$ and then also for all $s\in G$ (by continuity and denseness). Clearly, $f_\chi \neq 0$ if and only if there exists $g\in E$ with $a_g \neq 0$ and the last statement follows.
\end{proof}

\medskip
Next, we discuss the vanishing of Fourier--Bohr coefficients. Let us first introduce the following definition:

\begin{definition}
Let $K\subset \widehat{G}$ and $E\subset \mathcal{L}^2 (X, \mm)$ be given. We say that $E$ is {\rm perpendicular} to $K$, and write $E\perp K$, if  any element of $E$ is perpendicular to any eigenfunction to an eigenvalue from $K$.
\end{definition}

\begin{prop}[Vanishing result] \label{uniform-vanishing} Let $(X,G,  \mm)$ be a measurable dynamical system, let $\cA$ a F\o lner sequence and $G'$ a dense subgroup of $G$. Let $K\subset \widehat{G}$ be compact.

Let $E$ be  $G'$ invariant subspace of $\mathcal{L}^2 (X, \mm) \cap \mathcal{L}^\infty (X)$ with  a countable core and  $E\perp K$. Then,
\[
\lim_n \sup_{\chi \in K} \frac{1}{|A_n|} \left| \int_{A_n} f(t.x) \cdot \overline{\chi(t)} \dd t\right| = 0 \,,
\]
for any $f \in E$ and any $(E,G')$-generic point $x\in X$.
\end{prop}
\begin{proof} Let us assume that this does not happen. Then, there exists some $f \in E$, a $(E,G')$-generic point $x\in X$, a subsequence $(k_n)$ in the natural numbers, some $c>0$ and $\chi_n \in K$ such that
\[
\frac{1}{|A_{k_n}|} |\int_{A_{k_n}} f(t.x) \cdot \overline{\chi_n(t)} \dd t| \geq  c \, .
\]
Since $K$ is compact, without loss of generality we can assume $\chi_n \to \chi \in K$.

Next, let $D \subseteq E$ be a countable core. Then, for each $g \in D \subseteq \cL^\infty(X)$ the sequence
\[
\frac{1}{|A_{k_n}|} \int_{A_{k_n}} g(t.x)\cdot \overline{\chi_n(t)} \dd t
\]
is bounded. As $D$ is countable, via a standard diagonalisation argument, we can find a subsequence $(l_n)$ such that the following limit exists for all $g \in D$
\begin{equation}\label{eq1}
\lim_n \frac{1}{|A_{k_{l_n}}|} \int_{A_{k_{l_n}}} g(t.x) \cdot \overline{\chi_{l_n}(t)} \dd t \,.
\end{equation}
Since $D$ is a core in $E$, the limit in \eqref{eq1} then exists for all $g \in E$.

Moreover,
\[
\lim_n \frac{1}{|A_{k_{l_n}}|} \int_{A_{k_{l_n}}} f(t.x) \cdot \overline{\chi_{l_n}(t)} \dd t  \neq 0 \,.
\]
We can apply the previous lemma to conclude that there exists an eigenfunction to an eigenvalue from $K$ in $\overline{E}$. On another hand, by $E \perp K$ we immediately get  $\overline{E} \perp K$, a contradiction.
\end{proof}

Applying Proposition~\ref{uniform-vanishing} with  $K= \{ \chi \}$ we obtain the following consequence.

\begin{corollary}[Vanishing of Fourier--Bohr coefficient of non-eigenvalues for generic points] \label{cor-van}
 Let $(X,G,m)$ be a measurable dynamical system and $\cA$ a F\o lner sequence and $G'$ a dense subgroup of $G$. Let $\chi\in \widehat{G}$ not be an eigenvalue.   Let $E$ be   $G'$-invariant with a countable core. Then,
$$\lim_n \frac{1}{|A_n|} \int_{A_n}  f(t.x) \cdot \overline{\chi(t)}  \dd t =0$$
holds for any $(E,G')$-generic point $x \in X$ and any $f \in E$. \qed
\end{corollary}

\medskip
We can now prove the following result:

\begin{corollary}[Vanishing of Fourier--Bohr coefficients on countable sets] \label{cor-van-fb-d} Let $(X,G,\mm)$ be an ergodic dynamical system with second countable $G$. Let  $\cA =\{ A_n \}$ be a tempered F\o lner sequence.

Let $D \subseteq \mathcal{L}^2 (X, \mm)$ be a countable set. Then, there exists a set $X_0 \subseteq X$ of full measure with
\[
\frac{1}{|A_n|} \int_{A_n}  f(t.x) \cdot \overline{\chi(t)}  \dd t =0
\]
for all $f\in D$, $x\in X_0$  and $\chi \in \widehat{G}\setminus \SSS$.
\end{corollary}
\begin{proof}   As $G$ is second countable it possesses a countable dense subgroup $G'$.

Then,
\[
D':=\{ \tau_t f : f\in D, t\in G'\}
\]
is countable. Let $E$ be the linear span of $D'$. Then, $E$ is a subspace with a countable core that is $G'$-invariant.

By Proposition~\ref{existence-generic} we infer that there exists a set $X_0$ of full measure of $(E,G)$-generic points.

Now, the statement follows from the previous corollary.
\end{proof}

A consequence of the previous corollary   is the following.

\begin{theorem}[Wiener--Wintner] \label{thm-ww} Let $(X,G,\mm)$ be an ergodic  measurable  dynamical system with separable $L^2 (X,\mm)$ and second countable $G$ and let $\cA$ be a tempered  F\o lner sequence.

Let $E \subseteq \mathcal{L}^2 (X,\mm)$  be any subspace with countable core. Then, there exists a set $X_0 \subseteq X$ of full measure so that the following holds:
\begin{itemize}
  \item The Fourier--Bohr coefficient $a_{\chi}^\cA(f_x)$ exists for all $f \in E, x \in X_0$ and $\chi \in \widehat{G}$.
  \item  The Fourier--Bohr coefficient $a_{\chi}^\cA(f_x)$ is zero for all $f \in E, x \in X_0$ and all $\chi \notin \SSS$.
\end{itemize}
\end{theorem}
\begin{proof}
Let $D \subseteq E$ be a countable core.

By Corollary~\ref{cor-van-fb-d},  we can find a subset $X_1 \subseteq X$ of full measure so that the second bullet point holds for all $f \in D, x \in X_1$ and $\chi \notin \SSS$.

Next, for each $\chi \in \SSS$ we can pick an eigenfunction $f_\chi \in \cL^\infty(X)$. As $L^2(X, \mm)$ is separable, the set $\SSS$ is countable. Moreover, as $\mm$ is ergodic, each eigenspace is 1-dimensional. Therefore,
\[
\{ f_\chi : \chi \in \SSS \}
\]
is a basis for $L^2(X, \mm)_{\mathsf{pp}}$.

Hence, there exists a set $X_2 \subseteq X$ of full measure so that the pointwise  ergodic theorem holds for all $x \in X_2$ and each element from  $\{ h f_\chi : h \in D, \chi \in \SSS \}$.

The set $X_0:=X_1 \cap X_2$ satisfies the result.
\end{proof}


\section{The topological situation: $E = C(X)$}\label{sect:top}

Let us recall that a \textit{topological dynamical system} $(X,G)$ consists of a compact space $X$ together with a continuous  action $\alpha :  G\times X\longrightarrow X$ of $G$ on $X$.
Clearly, such an action is measurable with respect to the $\Sigma$-algebra of Borel sets. So, we are within the setting of the previous section.

We denote the set of continuous complex valued functions on a compact  topological space $X$ by $C(X)$ and equip it with the supremum norm
$$\|f\|_\infty :=\sup_{y\in Y} |f(y)| \,.$$

\medskip

The considerations of the previous section can  be applied with $E = C(X)$. Moreover, we can then formulate and prove uniform (in $x\in X$) results.

Note here that  $E = C(X)$ is a natural $G$-invariant set for which generic points may be considered. As the functions in $C(X)$ are bounded the vanishing requirement in the genericity definition is automatically satisfied. Moreover, as the constant function $g = 1$ belongs to $C(X)$ the genericity definition in the previous section agrees with the one usually given in the literature.

Specifically, whenever a topological  dynamical system $(X,G,)$, a F\o lner sequence
$(A_n)$ are given, and a $G$-invariant measure $\mm$ are given, a point $x\in X$ is easily be seen to be $(C(X),G)$-generic
with respect to the F\o lner sequence $(A_n)$  if and only if
\begin{equation}\label{eq-birk-et}
\lim_{n\to \infty} \frac{1}{|A_n|} \int_{A_n} f(t .x)\, \dd t =
\int_X  f(y) \, \dd \mm(y)
\end{equation}
holds for all continuous functions $f \in C(X)$. We refer to such points a \textit{generic for $(\mm,\mathcal{A})$} or, if $\mm$ and $\mathcal{A}$ are clear from the context, we just speak  about \textit{generic} points.

While the existence of generic points in general is not clear, when $\mm$ is ergodic the set of generic points along tempered F\o lner sequence has full measure by Proposition~\ref{existence-generic}. In particular, given any
ergodic measure $\mm$ and any F\o lner sequence $\cA$, there exists a subsequence $\cA'$ of $\cA$ such that the set of generic points for $(\mm,\mathcal{A})$ has full measure in $X$.

\medskip

Existence of a countable core was crucial for certain application in the last section. In the topological situation with $E = C(X)$, this is ensured by metrizability of $X$.  Indeed, as metrizability of $X$ is equivalent to $C(X)$ having a countable dense subset  (with respect to the supremum norm) we immediately infer the following statement.

\begin{prop}\label{prop4.1} Let $(X,G)$ be a topological dynamical system with metrizable $X$. Then, there exists a countable set $D \subseteq C(X)$ such that $D$ is a countable core for all $G$-invariant probability measures $\mm$ on $X$. \qed
\end{prop}

\begin{remark} Even if  $X$ is not metrizable, it is still possible that $C(X)$ has a countable core for specific measures $\mm$. This would hold for example if $\mm$ is supported on a finite orbit.
\end{remark}

Given this discussion, the results of the previous sections can all be applied with  $E = C(X)$ and $G' = G$  (and the additional assumption of metrizability of $X$ if necessary). We refrain from spelling this out here but rather  only state the following immediate consequence of Proposition~\ref{prop4.1}  and   Proposition \ref{uniform-vanishing}.

\begin{corollary}[Vanishing of Fourier--Bohr in generic points]  Let $(X,G)$ be a  topological dynamical system with metrizable $X$. Let $\mm$ be a $G$-invariant measure and let $\mathcal{A}$ be a F\o lner sequence.

Let $K$ be a compact subset of $\widehat{G}\setminus \SSS$. Then,
$$\lim_n \sup_{\chi \in K} \frac{1}{|A_n|} |\int_{A_n} f(t.x) \cdot \overline{\chi(t)} \dd t| = 0$$
for any $f\in C(X)$ and  any generic point $x\in X$. \qed
\end{corollary}

\begin{remark} In the case of metrizable $X$ and an ergodic measure $\mm$  it is a  rather direct  consequence of the classical Wiener--Wintner theorem that for a set of full measure we have vanishing of Fourier Bohr coefficients for all $f\in C(X)$  and $\chi \in \widehat{G} \backslash \SSS$.  However, to the best of our knowledge  it is a  new result that this set of points contains the  generic points.
\end{remark}

We now turn to a strengthening of the results of the previous section. This strengthening concerns a uniform (in $x\in X$) version of the results there. The key is the following topological variant of the basic observation.

\begin{prop}[Topological variant of the basic observation]\label{existence-m} Let $(X,G)$ be topological dynamical system and $\cA$ a F\o lner sequence. Assume that $X$ is metrizable.

Let $x_n  \in X$ and $\chi_n \in \widehat{G}$, be so that
\[
\lim_n \chi_n = \chi \in \widehat{G} \,.
\]
Then, there exists a subsequence $\cB=(B_n)_n$ of $\cA$, a $G$-invariant probability measure $\mm$ on $X$ and some $f_\chi \in L^2 (X, \mm)$ with  the following properties:
\begin{itemize}
  \item[(a)] For all $t \in G$ we have
  \[
  U_t f_\chi = \chi(t) f_\chi \,.
  \]
  \item[(b)] For all $f \in C(X)$ we have
  $$\lim_n \frac{1}{|B_n|} \int_{B_n} f(t.x_n) \dd t = \int_X f(y)  \dd \mm(y) \,.$$
  \item[(c)] For all $f \in C(X)$ we have
  $$\lim_n \frac{1}{|B_n|} \int_{B_n} f(t.x_n) \cdot \overline{\chi_n(t)} \dd t= \langle f, f_\chi\rangle \,.$$
\end{itemize}
\end{prop}
\begin{proof}
Let $D$ be a countable dense subset of $C(X)$. Then, for any F\o lner sequence $F_n$ and each $g \in D$, both sequences
\begin{align*}
&\frac{1}{|A_n|} \int_{A_n} g (t.x_n) \dd t \qquad \mbox{ and } \\
&\frac{1}{|A_n|}\int_{A_n} g (t.x_n) \cdot \overline{\chi_n (t)} \dd t
\end{align*}
are bounded by $\| g\|_\infty$. Therefore, $\{ A_n \}$ has subsequences for which these sequences  converge.
Via a standard diagonalisation argument, we can then find a subsequence $\cB=\{ A_{k_n} \}$ of $\cA$ such that for all $g \in D$ we have
\begin{itemize}
\item  the sequence $\frac{1}{|A_{k_n}|} \int_{A_{k_n}} g (t.x_{k_n}) \dd t$ converges to some number $L(g)$ and
\item  the sequence $\frac{1}{|A_{k_n}|} \int_{A_{k_n}} g (t.x_{k_n}) \cdot \overline{\chi_n(t)}  \dd t$ converges to some $a (g)$.
\end{itemize}
Now,  clearly the set of $g\in C(X)$ for which both these sequences converge  is  closed in $C(X)$.

Thus, we obtain that the sequence
\[
\frac{1}{|A_{k_n}|} \int_{A_{k_n}} f (t.x_{k_n}) \dd t
\]
converges to some number $L(f)$ for all $f \in C(X)$ and the sequence
\[
\frac{1}{|A_{k_n}|} \int_{A_{k_n}} f (t.x_{k_n}) \cdot  \overline{\chi_n(t)}  \dd t
\]
converges to some number $a(f)$ for all $f\in C(X)$.

It is clear that for all $f \in C(X)$ with $f \geq 0$ we have $L(f)\geq 0$ and that
\[
L(1_{X}) =1 \,.
\]
Thus, by the Riesz- representation theorem, there exists a Radon probability measure $\mm$  on $X$ so that
\[
L(f) = \int_{X} f(y) \dd y \qquad \forall f \in C(X) \,.
\]
The F\o lner property gives that $\mm$ is $G$-invariant. This proves (b).

Next, a standard application of Cauchy-Schwarz inequality (see proof of the lemma for the basic observation) gives
$$|a(f)|\leq \|f\|_2 \,.$$
It follows that $a$ is a linear bounded map on $C(X)\subset L^2 (X, \mm)$, and hence it can be extended to a bounded linear map
\[
a: L^2(X, \mm) \to \CC \,.
\]
Therefore, by Riesz- lemma, we can find an $f_\chi\in L^2 (X, \mm )$ such that, for all $f \in C(X)$ we have
$$a(f) = \langle f, f_\chi\rangle \,.$$
This gives (c).

A similar argument as in the proof of  Lemma~\ref{lem-eigen-exi} gives that $f_\chi$ satisfies
\[
U_s f_\chi = \chi(s) f_\chi \qquad \forall s\in G \,.
\]
This shows (a) and completes the proof.
\end{proof}

The proposition has the following consequence on vanishing of Fourier--Bohr coefficients uniform in $x\in X$:

\begin{lemma}[Uniform vanishing] Let $(X,G)$ be a dynamical system with metrizable $X$ and let $\cA$ be a F\o lner sequence. Let $\chi \in \widehat{G}$.
Then, the following assertions are equivalent:
\begin{itemize}
  \item[(i)] For each $G$-invariant measure $\mm$ on $X$, $\chi$ is not an eigenvalue for $\mm$.
  \item[(ii)] For each ergodic measure $\mm$ on $X$, $\chi$ is not an eigenvalue for $\mm$.
  \item[(iii)]  For each $f\in C(X)$ and any sequence $\chi_n \in \widehat{G}$ converging pointwise to $\chi \in\widehat{G}$  we have
\[
\lim_n \frac{1}{|A_n|} \int_{A_n}  f(t.x) \cdot \overline{\chi_n(t)} \dd t =0
\]
uniformly in $x\in X$.
\end{itemize}
\end{lemma}
\begin{proof}
(ii) $\Longrightarrow$ (i) follows from a standard application of Choquet Theorem \cite{Phe} (see \cite{FRS} for the details).

\smallskip

(i) $\Longrightarrow$ (iii) Assume by contradiction that there exists some  $f \in C(X)$  so that
\[
\frac{1}{|A_n|} \int_{A_n} f(t.x) \cdot \overline{\chi_n(t)} \dd t
\]
 does not converge to $0$ uniformly in $x\in X$. Then, we can find some $C>0$ and some $x_{k_n}$ in $X$ so that for all $n$ we have
\[
\left|\frac{1}{|A_{k_n}|} \int_{A_{k_n}} f(t.x_{k_n}) \cdot \overline{\chi_{k_n}(t)} \dd t \right| \geq C \,.
\]
As $\| f \|_\infty < \infty$, we can find a subsequence $(l_n)$ of $(k_n)$ such that
\[
\lim_n \frac{1}{|A_{l_n}|} \int_{A_{l_n}}  f(s.x_{l_n}) \cdot  \overline{\chi_{l_n}(s)}  \dd s  = a \neq 0 \,.
\]
Proposition~\ref{existence-m} implies that there exists a $G$-invariant probability measure $\mm$ so that $\chi$ is an eigenvalue of $\mm$. This is a contradiction to (i).

\smallskip

(iii) $\Longrightarrow$ (ii) Assume by contradiction that $\chi$ is an eigenvalue for some ergodic measure $\mm$. Pick a normalized eigenfunction $f_\chi$.
Chose some $f \in C(X)$ so that $\langle f, f_\chi \rangle \neq 0$.

Next, by Theorem~\ref{thm-point-erg}, we can pick a subsequence $ (A_{k_n})$ of $\cA$ along which the Birkhoff's ergodic theorem holds.

Then, there exists some set $X_0 \subseteq X$ with $\mm(X_0)=1$ so that for all $x \in X_0$ we have
\[
\langle f, f_\chi \rangle =f_{\chi}(x) \lim_n \frac{1}{|A_{k_n}|} \int_{A_{k_n}} f(t.x) \cdot \overline{\chi(t)}\dd t \, .
\]
Fix some $x \in X_0$. Then, (iii) gives
\[
\lim_n \frac{1}{|A_{k_n}|} \int_{A_{k_n}}f(t.x)\cdot \overline{\chi(t)} \dd t=0 \,.
\]
Therefore,
\[
0 \neq \langle f, f_\chi \rangle = f_\chi(x) \cdot 0 =0 \,,
\]
which is a contradiction.
\end{proof}

\begin{remark} To the best of the authors knowledge this is the first result on uniform (in $x\in X$) vanishing of Fourier--Bohr coefficients outside of the uniquely ergodic situation. For the uniquely ergodic situation we recover some previous results as discussed next.
\end{remark}

Recall that a dynamical system $(X,G)$ is \textit{uniquely ergodic} if it admits just one invariant probability measure.
For such systems, conditions (i) and (ii) of Lemma~\ref{uniform-vanishing} involves just one measure $\mm$ and we obtain the following consequence that  recovers a  theorem of \cite{Len}, which in turn  generalizes \cite[Theorem~1.1]{Rob} (for $G=\ZZ$) as well as a result of Assani \cite{Ass}.

\begin{corollary} Let $(X, G)$ be a uniquely ergodic metrizable dynamical system. Let $K$ be a compact subset of $\widehat{G}\setminus \SSS$. Then, for all $f \in C(X)$ we have

\[
\lim_n \sup_{x\in X} \sup_{\chi \in K} \frac{1}{|A_n|} \left| \int_{A_n} f(t.x) \cdot \overline{\chi(t)} \dd t \right| =0  \,.
\] \qed
\end{corollary}

\section{Wiener--Wintner points}\label{sec-Strongly}
In this section we introduce and study the main objects of interest for the remaining part of the article viz Wiener--Wintner points. We show that the Fourier--Bohr coefficients exist for such points and provide various characterizations for them.

\medskip

Consider a topological dynamical system $(X,G)$. Let a F\o lner sequence $\mathcal{A}$ be given.  We call an $x\in X$ a \textit{Fourier--Bohr point} if
\[
a_\chi^\cA(g_x) = \lim_n \frac{1}{|A_n|} \int_{A_n} g(t.x) \cdot \overline{\chi(t)} \dd t
\]
exists for all $g$ in $C(X)$ and all $\chi \in \widehat{G}$.

Given any $G$-invariant measure $\mm$, by Lemma~\ref{lem-eigen-exi}, any Fourier--Bohr point $x$ comes with a unique  system of functions $\{ f_\chi^x , \chi \in \widehat{G} \}$
satisfying
\[
a_\chi^{\cA}(g_x) = \langle g, f_\chi^x\rangle
\]
for all $g\in C(X)$.

Moreover, the non-vanishing $f_\chi^x, \chi \in \widehat{G}$,  form an orthogonal system of eigenfunctions.  In general these functions will not be normalized and not form a basis of the space of eigenfunctions as shown in the next example.

\smallskip

\begin{example} Consider $\{0,1\}$ with discrete topology. Then,  $Y =\{0,1\}^\ZZ$, equipped with product topology is a compact space.

$G =\ZZ$ acts continuously on $Y$ via
$$(\alpha(n,y)) (k) = y(k-n)$$ for $n\in\ZZ$ and $y\in Y$. Let
\[
A_n:=[-n,n] \,.
\]
Define
\[
u(n) :=
\left\{
\begin{array}{cc}
  0 & \mbox{ if } n < 0 \\
  1 & \mbox{ if } n \geq 0 \,.
\end{array}
\right.
\]
Then, $u \in Y$.

Define
\[
O(u):=  \{\alpha(n,u) : n \in \ZZ \}  \,,
\]
and $X:= \overline{O(u)}$. Then $(X, \ZZ)$ is a topological dynamical system. It is easy to see that
\[
X= \{\alpha(n,u) : n \in \ZZ \} \cup \{ \underline{0}, \underline{1} \}
\]
where $\underline{1}(n) =1$ for all $n\in \ZZ$ and the point  $\underline{0}$ with $\underline{0}(n) =0$ for all $n\in\ZZ$.   Both $\underline{1}$ and $\underline{0}$ are invariant under the action.

Then, the measure
$$\mm :=\frac{1}{2} \delta_{\underline{1}} + \frac{1}{2} \delta_{\underline{0}} \,,$$
is a $\ZZ$-invariant probability measure, where $\delta_a$ denotes the Dirac measure supported at $a$.

A short computation shows that every  $x\in O(u)$ is generic for $\mm$ and that the Fourier--Bohr coefficients exist for each $x\in X$. Moreover, the functions $f_\chi^x$ can be calculated explicitly and satisfy
\[
f_\chi^x=
\left\{
\begin{array}{cc}
 0 & \mbox{ if } \chi \neq 1 \\
 \frac{1}{2} 1_{\underline{0}} & \mbox{ if } \chi =1 \mbox{ and } x=  \underline{0} \\
  \frac{1}{2} 1_{\underline{1}} & \mbox{ if } \chi =1 \mbox{ and } x=  \underline{1} \\
    \frac{1}{2}(1_{\underline{1}} + 1_{\underline{0}}) & \mbox{ if } \chi =1 \mbox{ and } x \in O(u) \\
\end{array}
\right. \,.
\]
Therefore, $f_\chi^x$ are not normalized. Moreover, for each $x \in X$, the functions $f_{\chi}^x$ do not form a basis for the two dimensional space $L^2 (X,\mm)$.
\end{example}
\smallskip

The example suggests to single out those Fourier--Bohr points for which the non-zero $f_\chi^x , \chi \in \widehat{G}$  do form an orthonormal basis of the space of eigenfunctions. This leads to the following definition.

\begin{definition}[Wiener-Winter points] Let $(X,G)$ be a topological dynamical system, $\mm$ a $G$-invariant probability measure on $X$ and $\mathcal{A}$ a F\o lner sequence.

A point $x\in X$ is called a {\rm Wiener--Wintner point \ of $(X,G,\mm)$ with respect to $\mathcal{A}$} if
 there exists a set $B$ in $\mathcal{L}^2 (X, \mm)$ consisting of eigenfunctions with the following two properties:
 \begin{itemize}
   \item{}  For each $\chi \in \SSS$ the set $\{ e \in B : e \mbox{ is an eigenfunction for } \chi \}$ is an orthonormal basis in the eigenspace for $\chi$.
   \item{} For all $f \in C(X)$  and $e\in B$ with corresponding eigenvalue $\chi_e$ the following limit exists and
   \[
   \lim_n \frac{1}{|A_n|} \int_{A_n} f(t.x) \cdot \overline{\chi_e(t)} \dd t=\langle f, e\rangle \,.
   \]
 \end{itemize}
In this case  $B$ is called the {\rm basis associated to $x$}.
\end{definition}

We note that the Wiener-Wintner points are indeed exactly the Fourier-Bohr points  $x\in X$
 for which the non-zero $f_\chi^x , \chi \in \widehat{G}$, form an orthonormal basis of the space of eigenfunctions.
At this point it is not clear that Wiener--Wintner points exists at all. However, we will show in the subsequent considerations that existence of a Wiener--Wintner point implies ergodicity of the system, which in turn implies that almost every point is a Wiener--Wintner point. So, if such points exist at all there is an ample supply of them.

We start by  showing  that existence of Wiener--Wintner points implies ergodicity.

\begin{prop} Let $(X, G, \mm)$ be a dynamical system and $\cA$ a F\o lner sequence. If there exists an  Wiener--Wintner point  $x\in X$, then $\mm$ is ergodic.
\end{prop}
\begin{proof}
Assume by contradiction that $\mm$ is not ergodic. Then, there exist distinct $g,h \in B$ corresponding to the eigenvalue $1$. Then,  for all $f \in C(X)$ we get
$$
  \langle f, g \rangle = \lim_n \frac{1}{|A_n|} \int_{A_n} f(t.x) \dd t  = \langle f, h\rangle \,.$$
Since $C(X)$ is dense in $L^2 (X,\mm)$ this gives the contradiction $g = h$. 	
\end{proof}

By the preceding proposition we can (and will) from now on  write the basis $B$ appearing in the definition of generic points in the form
$$\{ e_\chi : \chi \in\SSS\},$$ where $e_\chi$ is the unique eigenfunction in  $B$  to $\chi$.  Next, we show that any Wiener--Wintner point is generic.

\begin{lemma}[Wiener--Wintner points are generic]\label{se-implies-gen} Let $(X, G)$ be a topological dynamical system, $\mm$ a $G$-invariant measure on $X$ and $\cA$ a F\o lner sequence. Then, every Wiener--Wintner point of $(X,G,\mm)$ with respect to $\cA$ is generic.
\end{lemma}
\begin{proof} Let  $x$ be a Wiener--Wintner point and let
 $e_1$ be the eigenfunction to the eigenvalue $1$ in the associated basis.  As existence of Wiener--Wintner points implies that the system is ergodic the function $e_1$ is almost surely constant with value, say, $c\neq 0$.  Since $x$ is Wiener--Wintner point we then get
\[
c = \langle 1, e_1 \rangle =  \lim_n \frac{1}{|A_n|} \int_{A_n} 1 \dd t =1 \,.
\]
So, we infer $e_1 =1$ almost everywhere.
The definition of Wiener--Wintner point then gives
\[
\int_{X} f(y) \dd \mm(y)= \langle f, e_1 \rangle  = \lim_n \frac{1}{|A_n|} \int_{A_n} f(t.x) \dd t
\]
for all $f \in C(X)$, and hence $x$ is generic.
\end{proof}

Since generic elements uniquely identify the measure for which they are generic we get the following.

\begin{corollary} Let $(X,G)$ be a topological dynamical system and $\cA$ a F\o lner sequence.  Then, any $x\in X$ can be a  Wiener--Wintner point for at most one invariant measure $\mm$.\qed
\end{corollary}

By combining Lemma~\ref{se-implies-gen} with Corollary \ref{cor-van} we can now infer that the Fourier--Bohr coefficients exist for Wiener--Wintner points.

\begin{theorem}[Existence of Fourier--Bohr coefficients for Wiener--Wintner points] Let $(X, G)$ be a topological dynamical system, $\mm$ a $G$-invariant measure on $X$ and $\cA$ a F\o lner sequence. Assume that the element $x \in X$ is Wiener--Wintner point for $(\mm, \cA)$ and  $B=\{ e_\chi : \chi \in \SSS\}$ be the associated  basis of eigenfunctions.
Then, for all $f \in C(X)$ and all $\chi \in \widehat{G}$ the Fourier--Bohr coefficient $a_{\chi}^\cA(f_x)$ exists and satisfies
\[
a_{\chi}^\cA(f_x) =
\left\{
\begin{array}{lc}
\langle f, e_\chi \rangle & \mbox{  if  } \chi \in \SSS \\
0 & \mbox{ if } \chi \notin \SSS
\end{array}
\right. \, .
\]
\qed
\end{theorem}

Having discussed various properties of Wiener--Wintner points we now turn to proving their existence.

\begin{prop}[Existence of Wiener--Wintner points]\label{existence-strong-gen} Let $(X, G, \mm)$ be an ergodic dynamical system and let $\cA$ be a F\o lner  sequence. Assume that $X$ is metrizable. Then, the following assertions hold:

\begin{itemize}
  \item[(a)]If  the pointwise ergodic theorem holds along $\cA$ then the set of Wiener--Wintner points has full measure in $X$.
  \item[(b)] $\cA$ has a subsequence $\cA'$ such that the set of Wiener--Wintner points  with respect to $\cA'$ has full measure.
\end{itemize}
\end{prop}
\begin{proof}
\textbf{(a)} Let $B \subseteq \cL^\infty(X)$ be an orthonormal basis of $L^2 (X,\mm)_{\textsf{pp}}$ consisting of eigenfunctions.

Adjusting the elements of $B$ on sets of measure zero we can assume without loss of generality that
$$e(t.x) = \chi_e (t) \cdot e (x)$$
holds for all $t\in G, e \in B$ and $x\in X$, where $\chi_e$ is the eigenvalue to $e\in B$.

Chose now  a  countable dense set $D \subseteq C(X)$. Then,
\[
E:= \{ f \cdot \overline{e} : f \in D, e\in B \}  \,.
\]
is countable. Moreover, all functions in $E$ are bounded. It follows that the pointwise  ergodic theorem holds simultaneously for all $g \in E$ on a full measure set $Z$. In particular we obtain for all $g\in E$ and $x\in Z$
\begin{eqnarray*}
\overline{e (x)} \lim_n \frac{1}{|A_n}| \int_{A_n} f(t.x) \cdot \overline{\chi_e (t)} \dd t
&= & \lim_n \frac{1}{|A_n}| \int_{A_n} f(t.x) \cdot \overline{e(x)} \cdot \overline{\chi_e (t)} \dd t\\
&=& \lim_n \frac{1}{|A_n}| \int_{A_n} f(t.x) \overline{e (t.x)}  \dd t\\
& = & \int_X f(y) \cdot \overline{ e(y)} \dd \mm (y) = \langle f, e\rangle \,.
\end{eqnarray*}

Moreover, as eigenfunctions have constant modulus and there are only countably many eigenfunctions, we can assume without loss of generality that $e(x)\neq 0$ holds for all $e\in B$ and $x\in Z$.

Hence, the above chain of equalities gives that
each  $x\in Z$ is a Wiener--Wintner point with associated basis given by 
$$e_\chi := \frac{1}{e(x)} e,$$
where $e$ is the unique eigenfunction in $B$ to the eigenvalue $\chi$.

\textbf{(b)} By  Theorem~\ref{thm-point-erg}, $\cA$ has a subsequence $\cA'$ along which the pointwise ergodic theorem holds. The claim follows now from (a).
\end{proof}

We finish this section with giving various characterizations of Wiener--Wintner points.

\begin{prop}\label{sg-char} Let $(X, G)$ be a topological dynamical system, $\mm$ a $G$-invariant probability measure on $X$ and $\cA$ a F\o lner sequence. The following assertions are equivalent for $x \in X$.
\begin{itemize}
  \item[(i)]  $x$ is Wiener--Wintner point for $(\mm, \cA)$ .
  \item[(ii)] $x$ is generic, and for all $f \in C(X)$ and $\chi\in \widehat{G}$ the Fourier--Bohr coefficient $a_\chi^\cA(f_x)$ exists and, for any eigenfunction $g_\chi$ with $\| g_\chi \|_2=1$ we have
  \[
  |a_\chi^\cA(f_x)| = |\langle f, g_\chi \rangle| \ .
  \]

    \item[(iii)] $x$ is generic, the dynamical system is ergodic and there exists a choice of non-trivial eigenfunctions $h_\chi$, $\chi  \in \SSS$,  such that\footnote{Note that if $h_\chi$ is an eigenfunction, then $|h_\chi|$ is constant $\mm$-almost everywhere. Therefore, for $\mm$-almost all $y \in X$ we have $|h_\chi(x)|=\|h_\chi\|_2$. The second condition below requires this equality to hold at the particular point $x$ we fixed.}:
  \begin{itemize}
    \item{} $h_\chi(t.x)=h_\chi(x) \cdot \chi(t)$ holds for all $t \in G$.
    \item{} $|h_\chi(x)|=\|h_\chi\|_2$.
    \item{} for all $f \in C(X)$ and all $\chi$ we have
  \[
  \langle f, h_\chi \rangle = \lim_n \frac{1}{|A_n|} \int_{A_n} f(t.x) \cdot  \overline{h_\chi(t.x)} \dd t \,.
  \]
  \end{itemize}
  \item[(iv)] The dynamical system is ergodic and there exists basis $f_\chi$ of normalized eigenfunctions for $L^2 (X,\mm)_{\mathsf{pp}}$ such that,
   for all $f \in C(X)$ all $N \in \NN$ and all $c_1, \ldots, c_N \in \CC, \chi_1, \ldots , \chi_N$ we have
  \[
\| f- \sum_{j=1}^N c_j f_{\chi_j} \|_2^2 =  \lim_n  \frac{1}{|A_n|} \int_{A_n} \left|f(t.x) - \sum_{j=1}^N c_j \chi_j(t)  \right|^2 \dd t  \,.
  \]
\end{itemize}
\end{prop}
\begin{remark} The condition (iv) is equivalent to
\[
\| f- \sum_{j=1}^N c_j f_{\chi_j} \|_2= \|f_x - \sum_{j=1}^N c_j \chi_j  \|_2
\]
and the fact that the semi-norm on the RHS is actually the limit (and not just limsup).
\end{remark}
\begin{proof}
\noindent (i) $\Longrightarrow$ (ii) Since $x$ is generic, the existence of Fourier--Bohr coefficients for $\chi \notin \SSS$ follows from Corollary~\ref{cor-van}.

Let $\chi \in \SSS$ and let $e_\chi \in B$ be the eigenfunction for $\chi$. Let $f \in C(X)$. Then, by definition of Wiener--Wintner point, the Fourier--Bohr coefficient $a^\cA_{\chi}(f_x)$ exists and
\[
a^\cA_\chi(f_x)=\langle f, e_\chi\rangle \,.
\]
Let $g_\chi$ be any eigenfunction for $\chi$ be so that $\| g_\chi \|_2=1$. Since the eigenspace of $\chi$ is one dimensional by ergodicity, $g_\chi=c e_\chi$ holds for a suitable $c\neq 0$. Since both $g_\chi, e_\chi$ are normal, $|c|=1$.

Therefore,
\[
|\langle f, g_\chi \rangle|=|c| \cdot \langle f, e_\chi \rangle= |a^\cA_\chi(f_x)| \,.
\]

\smallskip

\noindent (ii) $\Longrightarrow$ (i)
Consider $\chi \in\widehat{G}$.  By Lemma \ref{lem-eigen-exi} we can restrict to the case that $\chi $ is an eigenvalue.

Pick some normalized eigenfunction $g_\chi$ and some $f \in C(X)$ so that $\langle f, g_\chi \rangle \neq 0$.

By Lemma~\ref{lem-eigen-exi} applied to $f$ and the core $E=C(X)$, there exists an eigenfunction $h_\chi$ so that
\[
a_\chi(g_x)= \langle g, f_\chi \rangle \qquad \mbox{ for all } g \in C(X) \,.
\]


To finish this direction we show that the $(f_\chi)$ form a basis of $L^2 (X,\mm)_{\mathsf{pp}}$.

To do so it suffices to show that each eigenspace is one-dimensional. Assume that contrary and let $g,h$ be two orthogonal normalized  eigenfunctions to the same eigenvalue, say, $\chi$.

Then,
\[
|a_{\chi}^{\cA} (f_x)| = |\langle f, a g + b h\rangle| = |a \langle f,g\rangle + b \langle f,h\rangle|
\]
holds for all $f\in C(X)$ and all $a,b\in\CC$ with $|a|^2 + |b|^2 =1$.

Denoting by $z=  \langle f,g\rangle, w := \langle f,h\rangle$ we get that these two complex numbers satisfy
\[
|z|=|w|=|az+bw|
\]
for all $a,b \in \CC$ with $|a|^2+|b|^2=1$. It is straightforward to deduce that $z=w$.

This gives that
\[
\langle f,g\rangle = \langle f,h\rangle \,.
\]
Since this holds for all $f \in C(X)$, we get that $g=h$ which contradicts their orthonormality.


\smallskip

\noindent (i) $\Longrightarrow$ (iii) We first note  that the orbit $O_x=\{ t.x :t \in G\}$ is $G$-invariant. Therefore, by ergodicity $\mm(O_x) \in \{ 0,1 \}$. We split the problem into two cases:

{\it Case 1:} $\mm(O_x)=0$.

Then we can define
\[
h_\chi(y):=
\left\{
\begin{array}{cc}
  \chi(t) & \mbox{ if } y =t.x \in O_{x} \\
  f_\chi(y) &\mbox{ otherwise }
\end{array}
\right. \ .
\]
We show that this is well defined, and then it is immediate that $h_\chi$ is an eigenfunction satisfying the desired condition.
Pick one $f \in C(X)$ is such that $\langle f, f_\chi \rangle \neq 0$.
Assume that some $y \in O_{x}$ has two representations
\[
y=s.x = r.x \ .
\]
Then, the function $f_x$ is $s-r$ periodic. This gives
\[
\langle f, f_\chi \rangle = a_{\chi}^\cA(f_x)= a_{\chi}^\cA( ((U_{s-r} f)_x)= \overline{\chi(s-r)} \cdot a_{\chi}^\cA(f_x)= \overline{\chi(s-r)} \cdot a_{\chi}^\cA(f_x) \cdot \langle f, f_\chi \rangle \ .
\]
As $\langle f, f_\chi \rangle \neq 0$ we get $\chi(s)=\chi(r)$, completing the proof in this case.

\underline{Case 2:} $\mm(O_{x})=1$.

On $O_{x}$ the operation
\[
(t.x) \oplus (s.x) := (s+t).x
\]
defines a group structure ( and it is immediate that $(O_{x}, \oplus)$ is isomorphic as a topological group to $G/\Per(x)$).

It follows immediately that any $G$-invariant measure on this group is the Haar measure. Therefore, the ergodic measure is the unique probability Haar measure on $O_x$ and $O_x$ is a compact group.

Now, for each $f \in C(X)$ the function $f_x$ is $\Per(x)$-periodic and hence the Fourier--Bohr spectrum $\{ \chi : a_\chi(f_x) \neq 0 \}$ of $f_x$ is contained in
\[
(\Per(x))^0 \subseteq \widehat{G} \ .
\]
It follows immediately from here that for each eigenvalue $\chi$ the function we can define an eigenfunction
\[
h_\chi(y):=
\left\{
\begin{array}{cc}
  \chi(t) & \mbox{ if } y =\tau_t x \in O_{x} \\
  0 &\mbox{ otherwise }
\end{array}
\right. \ .
\]
continuous on the set $O_x$ of measure $0$.

Since $(O_x, G, \mm)$ is uniquely ergodic, and $h_\chi$ is continuous on $O_x$, we can apply the unique ergodic theorem on $O_x$ to get that for all $f \in C(X)$ and all eigenvalues $\chi$ we have
\[
\langle f, h_\chi \rangle =a_\chi^\cA(f_x) = \langle f, f_\chi \rangle \ .
\]
This implies that $\| f_\chi-h_\chi \|_2=0$ and hence $\| h_\chi \|_2=\|f_\chi\|_2=1$ and $h_\chi$ are indeed normalized.

\medskip
\noindent (iii) $\Longrightarrow$ (i) Let $h_\chi$, $\chi \in \SSS$,  be any such choice. Pick one $f \in C(X)$ such that
\[
\langle f, h_\chi \rangle \neq 0 \ .
\]
This implies that
\[
0\neq \lim_n \frac{1}{|A_n|} \int_{A_n} f(t.x) \cdot \overline{h_\chi(t.x)} \dd t
\]
and hence $h_\chi(x) \neq 0$.
Defining
\[
f_\chi:= \frac{1}{h_\chi(x)} h_\chi \ .
\]
we get a basis of normalized eigenfunctions associated to $x$ and $x$ is Wiener--Wintner point.

\noindent (i) $\Longrightarrow$ (iv) is immediate. Indeed
\begin{align*}
\| f- \sum_{j=1}^N c_j f_{\chi_j} \|^2_2 &= \int_{X} |f(y)|^2 \dd \mm(y)- \sum_{j=1}^N c_j \int_{X} \overline{ f(y)} f_{\chi_j}(y) \dd \mm(y)\\
&- \sum_{j=1}^N \bar{c_j} \overline{ \int_{X} \overline{ f(y)} f_{\chi_j}(y) \dd \mm(y)} + \sum_{j=1}^N |c_j|^2 \ .
\end{align*}
Since $f\in C(X)$ we can apply genericity for $x$ and $|f|^2$ and the definition of $x$ being Wiener--Wintner point for $\bar{f}$ to get
\begin{align*}
\| f- \sum_{j=1}^N c_j f_{\chi_j} \|^2_2&= \lim_n \frac{1}{|A_n|} \int_{A_n} \left(|f_x(t)|^2 - \sum_{j=1}^N c_j \overline{f_x(t)} \chi_j(t) \right. \\
&\left. -\sum_{j=1}^N \bar{c_j} f_x(t) \overline{\chi_j(t)} + \sum_{j=1}^N |c_j|^2 \right) \dd t\\
&= \lim_n  \frac{1}{|A_n|} \int_{A_n} \left|f_x(t) - \sum_{j=1}^N c_j \chi_j(t)  \right|^2 \dd t \ .
\end{align*}

\medskip (iv) $\Longrightarrow$ (i)
Note first that for each $f \in C(X), c \in \CC$ and $\chi$ we have
\begin{align*}
&\int_{X} |f(y)|^2   -   \overline{c} f(y) \overline{f_\chi(y)} - c \overline{f(y)} f_\chi(y) +|c|^2 \dd \mm(y) =\| f-  c f_{\chi} \|^2_2\\
&= \lim_n  \frac{1}{|A_n|} \int_{A_n} |f(t.x)|^2 -\overline{c} f(t.x)\overline{\chi(t)} - c \overline{f(t.x)} \chi(t) + |c|^2 \dd t \,.
\end{align*}
Applying this relation to $c=\pm 1, \pm i$ and taking a linear combination, via polarization we infer that $x$ is Wiener--Wintner point.
\end{proof}

\section{Relationship to  Besicovitch almost periodicity}
\label{sec-Besicovitch}

The aim of this section is to relate Wiener--Wintner points to Besicovitch almost periodicity.

\medskip

We start by briefly recalling relevant definitions concerning Besicovitch almost periodicity. A trigonometric polynomial (on $G$) is a linear combination of elements of $\widehat{G}$. The set of all trigonometric polynomials on $G$ is denoted by $TP(G)$.

Let $\cA$ be a F\o lner sequence. Then, we define the associated \textit{Besicovitch seminorm}  on the bounded measurable functions $h: G\longrightarrow \CC$ by
$$\|h\|_{b} :=\limsup_n \sqrt{\frac{1}{|A_n|} \int_{A_n} |h(t)|^2 \dd t} \,.$$
We note  that this semi-norm is  denoted by $\| \, \|_{b,2,\cA}$ in the systematic studies in \cite{ LSS,LSS2}, but for our purposes here the simpler notation suffices.

A bounded measurable  $f : G\longrightarrow \CC$ is called \textit{Besicovitch almost periodic} (w.r.t. $\cA$)  if to any $\varepsilon >0$ there exists a $p\in TP(G)$  with
$$\|f - p\|_{b} \leq \varepsilon \,.$$
The set of all Besicovitch almost periodic functions with respect to $\cA$ is denoted by $\bap^2_{\cA}(G)$.

If a dynamical system $(X,G)$ and a F\o lner sequence $\cA$  is given  then an $x\in X$ is called \textit{Besicovitch almost periodic} (w.r.t. $\cA$)   if $f_x$ is Besicovitch almost periodic for all $f\in C(X)$ (compare \cite{LSS2} for a study of this concept and further details).  We note that this definition does not involve a measure. However, as discussed in \cite{LSS2} any Besicovitch almost periodic point induces a unique ergodic measure $\mm$ on $X$ for which it is generic and, moreover, the dynamical system $(X,G,\mm)$ has pure point spectrum.  Here is the main result of this section (compare \cite{LSS2}).

\begin{theorem}[Wiener--Wintner points as Besicovitch almost periodic points]\label{sg-vs-bes} Let $(X,G)$ be a topological dynamical system and $\cA$ a F\o lner sequence. Let $x \in X$ be arbitrary. Then, the following assertions are equivalent for $x\in X$:

\begin{itemize}
\item[(i)]  $x$ is a Besicovitch almost periodic point (along $\cA$).
\item[(ii)] There exists an ergodic measure $\mm$ with pure point spectrum such that $x$ is Wiener--Wintner point of $(X,G,\mm)$ with respect to $\cA$.
\end{itemize}
\end{theorem}

The proof needs some preparation.   Let us start by giving  a formula which gives the distance from $f \in C(X)$ to $(L^2(X,\mm))_{\mathsf{pp}}$ as a formula on $G$ in terms of $f_x$.

\begin{prop}\label{dist-pp} Let $(X,G,\mm)$ be a topological dynamical system and $\cA$ a F\o lner sequence and $x$ a Wiener--Wintner point of $X$. Let
$f \in C(X)$ be given and denote its projection onto $(L^2(X,\mm))_{\mathsf{pp}}$ by $f_{\mathsf{pp}}$. Then,
$$
 \|f - f_{\mathsf{pp}}\|_2 = \inf\{ \| f_x  - p \|_{b} : p \in TP (G) \}\\
 = \inf\{ \| f_x - h \|_{b} : h \in \bap^2_{\cA}(G) \} \,.$$

\end{prop}
\begin{proof}
The equality
\[
\inf\{ \| f_x - p \|_{b} : p \in TP(G) \} =\inf\{ \| f_x  - h \|_{b} : h \in  \bap^2_{\cA}(G) \}
\]
follows immediately from the definition of Besicovitch almost periodicity.
The equality
\[
\inf\{ \| f_x  - p \|_{b} : p \in TP(G) \} = \| f - f_{\mathsf{pp}} \|_2
\]
follows directly from the characterisation in Prop.~\ref{sg-char}(iv).
\end{proof}

The preceding proposition has the following consequence.

\begin{prop}\label{finpp} Let $(X,G,\mm)$ be an ergodic topological  dynamical system. Assume that there exist Wiener--Wintner points w.r.t. $\cA$.

Let $f \in C(X)$. Then, the following assertions  are equivalent:
\begin{itemize}
  \item[(i)] $f \in (L^2(X,\mm))_{\mathsf{pp}}$.
  \item[(ii)] There exists a Wiener--Wintner point $x$  in $(X,G,\mm)$ with respect to $\cA$  with $f_x \in \bap^2_{\cA}(G)$.
\item[(iii)] For every Wiener--Wintner point $x$ in $(X,G,\mm)$ with respect to $\cA$ we have $f_x \in \bap^2_{\cA}(G)$.
\end{itemize}
\end{prop}
\begin{proof}
(ii) $\Longrightarrow$ (i) $\Longrightarrow$  (iii) follows from Proposition~\ref{dist-pp}.

(iii) $\Longrightarrow$ (ii) holds as long as there are Wiener--Wintner points.
\end{proof}

\begin{proof}[Proof of Theorem~\ref{sg-vs-bes} ]
$(ii)\Longrightarrow(i)$ follows from Proposition~\ref{finpp}.

\smallskip

$(i)\Longrightarrow(ii)$ Since $x$ is Besicovitch almost periodic we have $f_x \in Bap^2_{\cA}(G)$ for all $f \in C(X)$. In particular
$a_{\chi}^\cA(f_x)$ exists for all $f \in C(X)$ and $\chi \in \widehat{G}$.

Next, note that
\[
\mm(f):= a_0^\cA(f_x)
\]
defines a positive linear mapping on $C(X)$ and hence a positive measure. As $\mm(1)=1$ it is probability. It is obviously $G$-invariant. Moreover, $x$ is a generic element for $\mm$.

Therefore, we are in the situation of Lemma~\ref{lem-eigen-exi}. Applying this result, we get that for each
\[
\chi \in \SSS':= \{ \chi \in \widehat{G} : \mbox{ there exists }  f \in C(X) \mbox{ with } a_\chi^\cA(f_x) \neq 0 \}
\]
there exists an eigenfunction $f_\chi$ with $\| f_ \chi \|_2 \leq 1$ which satisfies
$$\langle f, f_\chi \rangle =a_\chi^\cA(f_x) \qquad \mbox{ for all } f \in C(X).$$

It remains to show  that $\{ f_\chi : \chi \in \widehat{G} \}$ is an orthonormal basis for $L^2(X, \mm)$.

By the density of $C(X)$ in $L^2(X,\mm)$ it suffices to show that each $f \in C(X)$ belongs to the subspace spanned by these $f_\chi$.

Let $f \in C(X)$ be arbitrary. Then $f_x \in \bap_{\cA}(G)$. Let $\chi_n$ be an enumeration of the characters where $a_{\chi}^\cA(f_x) \neq 0$.
Let
\[
P_N= \sum_{k=1}^N a_{\chi_k}^\cA(f_x) \chi_k
\]
Then,
\[
\lim_N \| f_x-P_N \|_{b} =0 \,.
\]
Now, either via a direct expansion and computation, or via the Parseval identity for Besicovitch almost periodic functions, we have
\begin{align*}
 \| f_x-P_N \|^2_{b}&= \left( \lim_n \frac{1}{|A_n|} \int_{A_n} |f_x(t)|^2 \dd t \right) - \sum_{k=1}^N |a_{\chi_k}^\cA(f_x)|^2
\end{align*}
Define
\[
g_N:= \sum_{k=1}^N \langle f, f_{\chi_k} \rangle f_{\chi_k}= \sum_{k=1}^N a_{\chi_k}^\cA(f_x) f_{\chi_k}\,.
\]
Then,
\begin{align*}
 \| f-g_N \|_{2}^2&= \int_{X} |f(y)|^2 \dd \mm(y) - \sum_{k=1}^N \overline{a_{\chi_k}^\cA(f_x)} \langle f, f_\chi \rangle \\
 &-\sum_{k=1}^N a_{\chi_k}^\cA(f_x)\overline{ \langle f, f_\chi \rangle} + \sum_{k=1}^N |a_{\chi_k}^\cA(f_x)|^2 \| f_\chi \|^2_2 \\
 &= \int_{X} |f(y)|^2 \dd \mm(y) - 2 \sum_{k=1}^N |a_{\chi_k}^\cA(f_x)|^2 + \sum_{k=1}^N |a_{\chi_k}^\cA(f_x)|^2 \| f_\chi \|^2_2 \\
 &\leq \int_{X} |f(y)|^2 \dd \mm(y) -  \sum_{k=1}^N |a_{\chi_k}^\cA(f_x)|^2 \,,
\end{align*}
with the last inequality following from $\|f_\chi \|_2 \leq 1$.

Finally, as $x$ is generic for $\mm$ we have
\[
\int_{X} |f(y)|^2 \dd \mm(y) = \lim_n \frac{1}{|A_n|} \int_{A_n} |f_x(t)|^2 \dd t \ .
\]
Therefore,
\[
 \| f-g_N \|_{2} \leq \| f_x-P_N \|_{b} \,.
\]
As $\lim_N \| f_x-P_N \|_{b} =0$ we get that
\[
f= \sum_{k=1}^\infty \langle f, f_{\chi_k} \rangle f_{\chi_k} \ .
\]
This gives that $f$ is indeed in the pure point spectrum, and also that the eigenfunctions must be normalized. Therefore, $\mm$ has pure point spectrum (and is ergodic since our basis only contains one eigenfunction for $\chi=0$ by construction).

Since with respect to this basis of eigenfunctions we have
\[
a_\chi^\cA(f_x)=\langle f, f_\chi \rangle \qquad \mbox{ for all } f \in C(X) \,,
\]
$x$ is Wiener--Wintner point. The existence of Wiener--Wintner points implies (again) ergodicity.
\end{proof}

Let us complete the section by showing the equivalence between Wiener--Wintner points and the set of points which are generic and Besicovitch almost periodic.
This set of points has recently been a focus of attention  in \cite[Theorem~2]{KK}.

\begin{theorem} Let $(X, G, \mm)$ be an ergodic topological  dynamical system with pure point spectrum, let $\cA$ be a F\o lner sequence and let $x \in X$. Then, $x$ is a Wiener--Wintner point if and only if $x$ is $\mm$-generic and  Besicovitch almost periodic.
\end{theorem}
\begin{proof}
Let $x$ be a Wiener-Winter point. Then, $x$ is generic and, by  Theorem~\ref{sg-vs-bes}, it is Besicovitch almost periodic.

Conversely, if $x$ is generic Besicovitch almost periodic then  by Theorem~\ref{sg-vs-bes} $x$ is a Wiener--Wintner point  for some ergodic measure $\mm'$ with pure point spectrum.
It follows that $x$ is a generic point for both $\mm$ and $\mm'$ we get that $\mm=\mm'$. As $x$ is Wiener--Wintner point for $\mm'$, the claim follows.
\end{proof}

\section{Application to translation bounded measure dynamical systems}\label{sec-Application}
In this  section we give an  application to dynamical systems of translation bounded measures. Such dynamical systems serve as  models in aperiodic order. For a brief overview of the field of aperiodic order we recommend the monographs \cite{TAO,KLS}, while for an introduction to dynamical systems of translation bounded measures see \cite{BL}.

\smallskip

Starting from a Delone set $\Lambda$, or more generally a translation bounded measure $\mu$, the idea is to construct a dynamical system $\XX$ consisting of all pointsets/measures which locally look similar to $\Lambda$ or $\mu$, respectively. We will make this precise below. Then, via the so called Dworkin argument \cite{Dwo}, we can study the diffraction of $\Lambda$ or $\mu$ by studying the spectrum of the induced unitary representation of $G$ on $L^2(\XX, \mm)$ for some (ergodic) measure $\mm$ on $\XX$.
Let us now recall some of the basic definitions.

\smallskip

Recall that, via the Riesz- representation theorem \cite{Rud}, we can identify the set of \textit{Radon measures} on $G$ with the functionals on $\Cc(G)$ which are continuous in the so called inductive topology of $\Cc(G)$.

A  measure $\mu$  on $G$  is called \textit{translation bounded}  if the for all $\varphi \in \Cc(G)$ the convolution $\mu \ast \varphi : G\longrightarrow  \CC$ defined by
$$\mu \ast \varphi (t) :=\int_G \varphi (t-s) d\mu (s)$$
is a bounded function.  This is equivalent \cite{ARMA1} to
\[
\| \mu \|_{U} := \sup_{t \in G} \left| \mu \right|(t+U)  < \infty\,,
\]
for one (and hence all \cite{BM}) precompact set $U$ with non-empty interior. Here, $\left| \mu \right|$ denotes the total variation measure of $\mu$, see \cite[Theorem 6.5.6]{Ped}.

The set of all translation bounded measures is denoted by $\cM^\infty (G)$.

\smallskip

Given a precompact open set $U$ and a constant $C>0$ the set
\[
\cM_{C,U}:=\{ \mu \in \cM^\infty(G): \| \mu \|_{U} \leq C \}
\]
is vaguely compact and $G$-invariant \cite{BL}. Moreover, the action of $G$ on $\cM_{C,U}$ via translates is continuous in the \textit{vague topology} (i.e. the topology of convergence on test functions $\varphi \in \Cc(G)$). Therefore, any $G$-invariant vaguely closed set $\XX \subseteq \cM_{C,U}$ gives rise to a topological dynamical system $(\XX,G)$. Such a dynamical system is called a \textit{dynamical system of translation bounded measures} or just \textit{TMDS} for short.

Such systems can be thought of as generalizations of subshifts over finite alphabets. Indeed, for $G =\ZZ$ the translation bounded measures can be identified with the bounded functions.

In particular, for each $\mu \in \cM^\infty(G)$, the \textit{hull}
$$\XX(\mu) :=\overline{\{ \alpha(t,\mu) :t\in G\}}$$
is a TMDS.

\smallskip

Over the years much effort has been put into understanding the case of systems with pure point spectrum. The equivalence of pure point dynamical and diffraction spectrum was established in \cite{BL,LMS,Gou-1}. Recently, the equivalence between pure point spectrum for dynamical systems and mean/Besicovitch almost periodicity of its points has been established in \cite{LSS,LSS2}.

\smallskip

Finally,  we show that for TMDS, we can classify the Wiener--Wintner points in terms of Besicovitch almost periodicity of the point (considered as a translation bounded measure).

First, us recall  that a measure $\mu$ is called \textit{Besicovitch almost periodic} with respect to the van Hove sequence $(A_n)$ if $\mu\ast \varphi \in \bap_{\cA}^2(G)$ for all $\varphi \in \Cc (G)$.

\begin{corollary}[Characterization of Besicovitch almost periodic measures] Let $\mu \in \cM^\infty(G)$ and let $\cA$ be a van Hove sequence. Then,  $\mu$ is a Besicovitch almost periodic along $\cA$ if and only if there exists an ergodic measure $\mm$ on $\XX(\mu)$ with pure point spectrum such that $\mu$ is Wiener--Wintner point for $(\mm,\cA)$.
\end{corollary}
\begin{proof}
This is an immediate consequence of Theorem~\ref{sg-vs-bes} and \cite[Prop.~7.1]{LSS2}.
\end{proof}

\subsection*{Acknowledgments}
NS was supported by the Natural Sciences and Engineering Council of Canada via grant 2024-04853, and he is grateful for the support. The authors would like to
thank Michael Baake, Christoph Richard and Gerhard Keller on some illuminating discussions. The authors gratefully acknowledge BIRS and the 2025 workshop "Directions in Aperiodic Order" there during which the article was finalized.

\begin{appendix}

\section{Tempered F\o lner sequences}\label{tempered}

\begin{lemma}\label{lem:fol} Let $(A_n)$ be a F\o lner sequence and let  $B_n \subseteq G$, $n\in \NN$, be  given with
\[
\lim_n \frac{|A_n \Delta B_n|}{|A_n|} = 0 \,.
\]
Then, $(B_n)$ is F\o lner.  Moreover, if $(A_n)$ is tempered and $B_n \subseteq A_n$ holds for all $n\in \NN$,  then $(B_n)$ is also tempered.
\end{lemma}
\begin{proof} The assumptions on $(A_n)$ and $(B_n)$ easily  imply
\[
\lim_n \frac{|B_n|}{|A_n|}=1 \,.
\]
Now, for each $t \in G$ we clearly  have
\begin{align*}
  B_n \Delta (t+B_n) & \subseteq \left( B_n \Delta A_n \right) \cup \left( A_n \Delta (t+A_n) \right) \cup (t+(A_n \Delta B_n)) \,.
\end{align*}
This easily  gives that  $(B_n)$ is F\o lner. Next, if $B_n \subseteq A_n$ for all $n$, we then have
\begin{align*}
  \left| \bigcup_{k=1}^{n-1} (B_n-B_k) \right| & \leq   \left| \bigcup_{k=1}^{n-1} (A_n-A_k) \right| \\
   &\leq C |A_n| \leq C'|B_n|
\end{align*}
for some constant $C'$, with the last inequality following from the fact that $\frac{|B_n|}{|A_n|}$ converges to $1$ and hence is bounded.
\end{proof}

Let us note here in passing that if $B_n \subseteq A_n$ and $B_n$ is tempered, it does not seem that $A_n$ is tempered. This is the reason why below we need a separate proof for the union.

\begin{lemma}\label{Folner2} Let $(A_n)$ be a tempered F\o lner sequence and let $s \in G$. Then, the sequences $(B_n)$ and $(C_n)$ defined by
\begin{align*}
  B_n & =A_n \cap (s+A_n)\\
  C_n & =A_n \cup (s+A_n)\\
 \end{align*}
 are tempered van F\o lner sequences.
\end{lemma}
\begin{proof}
Lemma~\ref{lem:fol} implies that $(B_n)$ and $(C_n)$ are F\o lner sequences and that $(B_n)$ is tempered. For $C_n$ we obviously have

\begin{align*}
  \bigcup_{k=1}^{n-1} (C_n-C_k) 	& \subseteq \left( \bigcup_{k=1}^{n-1} (A_n-A_k) \right) \cup \left( \bigcup_{k=1}^{n-1} (A_n-(s+A_k)) \right) \\
& \cup \left( \bigcup_{k=1}^{n-1} ((s+A_n)-A_k) \right) \cup \left( \bigcup_{k=1}^{n-1} (s+A_n)-(s+A_k) \right) \,.
\end{align*}
This implies
\[
\left|    \bigcup_{k=1}^{n-1} (C_n-C_k)  \right| \leq 4 \left|    \bigcup_{k=1}^{n-1} (A_n-A_k)  \right|  \leq 4C |A_n| \,.
\]
giving that $(C_n)$ is tempered.
\end{proof}

\end{appendix}

\end{document}